\documentclass[12pt]{article}
\usepackage[utf8]{inputenc}

\usepackage{geometry,amsmath,mathtools,blkarray,amssymb,amsthm,setspace,xcolor,bbm,tikz-cd,xparse,amsfonts,authblk,graphics,datetime,thmtools,bm}

\usepackage[normalem]{ulem}

\usepackage[shortlabels]{enumitem}
\usepackage[colorlinks,linkcolor=blue,citecolor=blue,urlcolor=blue,bookmarks=false,hypertexnames=true]{hyperref}
\usepackage[
    backend=biber,
    style=bwl-FU, 
    sortcites=false, 
    maxcitenames=2, 
    mincitenames=1, 
    maxbibnames=8, 
    uniquelist=false
]{biblatex}

\usepackage{tikz-cd}
\usetikzlibrary{decorations.text}
\usepackage{tikz,tikz-3dplot}    
\usetikzlibrary{decorations.pathreplacing}

\allowdisplaybreaks

\providecommand{\keywords}[1]
{
  \small	
  \textbf{{Keywords---}} #1
}

\theoremstyle{plain}
\newtheorem{theorem}{Theorem}[section]

\newtheorem{lemma}{Lemma}[section]
\newtheorem{corollary}{Corollary}[section]

\theoremstyle{definition}

\renewcommand\thmcontinues[1]{Continued}

\newcommand{\tr}{\mathrm{tr}}
\newcommand{\norm}[1]{\left\|#1\right\|}
\newcommand{\ip}[2]{\left<#1, #2\right>}

\NewDocumentCommand{\defmathletter}{m}{%
    \expandafter\newcommand\csname b#1\endcsname{\mathbb{#1}}%
    \expandafter\newcommand\csname c#1\endcsname{\mathcal{#1}}%
}
\NewDocumentCommand{\defmathletters}{>{\SplitList{,}}m}{\ProcessList{#1}{\defmathletter}}
\defmathletters{A,B,C,D,E,F,G,H,I,J,K,L,M,N,O,P,Q,R,S,T,U,V,X,Y,Z}

\NewDocumentCommand{\defvector}{m}{%
    \expandafter\newcommand\csname v#1\endcsname{\mathbf{#1}}%
}
\NewDocumentCommand{\defvectors}{>{\SplitList{,}}m}{\ProcessList{#1}{\defvector}}
\defvectors{A,B,C,D,E,F,G,H,I,J,K,L,M,N,O,P,Q,R,S,T,U,V,X,Y,Z}
\defvectors{a,b,c,d,e,f,g,h,i,j,k,l,m,n,o,p,q,r,s,t,u,v,x,y,z}

\newdateformat{monthyeardate}{\monthname[\THEMONTH] \THEYEAR}

\title{Well-Posedness of Second-Order Uniformly Elliptic PDEs with Neumann Conditions}
\author{Haruki Kono \thanks{Email: hkono@mit.edu.}}
\affil{MIT}
\date{\monthyeardate\today}

\begin{document}

\maketitle

\begin{abstract}
Extending the results of \cite{nardi2015schauder}, this note establishes an existence and uniqueness result for second-order uniformly elliptic PDEs in divergence form with Neumann boundary conditions.
A Schauder estimate is also derived.
\end{abstract}


\section{Introduction}
For a $C^{2, \alpha}$-domain $\Omega \subset \bR^d,$ we consider 
\begin{equation} \label{eq:main-pde}
    \begin{cases}
        \nabla \cdot (A \nabla u) = f \text{ in } \Omega \\
        \ip{A \nabla u}{\vn} = g \text{ on } \partial \Omega
    \end{cases}
    ,
\end{equation}
where $f \in C^{0, \alpha} (\bar \Omega),$ $g \in C^{1, \alpha} (\bar \Omega),$ and $\vn : \partial \Omega \to \bR^d$ is the outward normal unit vector of $\partial \Omega.$
Throughout this note, we assume that $A = (a^{ij}) : \bar \Omega \to \bR^{d \times d}$ is a matrix valued function that satisfies $a^{ij} \in C^{1, \alpha} (\bar \Omega)$ and is uniformly elliptic, i.e., there exists $\lambda > 0$ such that $v^\prime A (x) v \geq \lambda$ for $x \in \Omega$ and $v \in \bR^d$ with $\norm{v} = 1.$
The symmetric part of $A$ defined as $A_S \coloneqq (A + A^\prime) / 2$ admits the eigendecomposition $A_S = \sum_i \lambda_i q_i q_i^\prime,$ where $\lambda_i \geq \lambda > 0$ and $q_i$ forms an orthogonal basis of $\bR^d.$

The PDE (\ref{eq:main-pde}) arises ubiquitously in applied science fields, including thermodynamics \cite{carslaw1959heat}, image processing \cite{weickert1998anisotropic}, and control theory \cite{sokolowski1992introduction}.
Its well-posedness is well known among experts but has not been explicitly stated in the literature, whereas results for second-order elliptic PDEs with Dirichlet and oblique boundary conditions are readily available in standard textbooks such as \cite{gilbarg1977elliptic}.
For the Laplace operator, i.e., the case of $A = I$, \cite{nardi2015schauder} establishes existence and uniqueness and derives a Schauder estimate.
In this note, we aim to provide an elementary proof of these results for second-order uniformly elliptic PDEs in divergence form, by extending the ideas developed by \cite{nardi2015schauder}.

\section{Existence and Uniqueness Result}
Before stating our main result, we shall show several auxiliary propositions.

\begin{lemma} \label{lem:second-derivative-is-definite}
Let $\Omega \subset \bR^d$ be a $C^2$-domain.
If $u \in C^2 (\bar \Omega)$ satisfies $u (p) = \max_{\bar \Omega} u \ (\text{resp.} \ \min_{\bar \Omega} u)$ and $\nabla u (p) = 0$ for $p \in \partial \Omega,$ then $D^2 u (p)$ is negative (resp. positive) semi-definite.
\end{lemma}

\begin{proof}
We shall show only the case where $p$ attains the maximum.
It suffices to show $v^\prime D^2 u (p) v \leq 0$ for each $v \in \bR^d \setminus \{0\}.$
We consider two distinct cases.

First, suppose that $v \notin T_p \partial \Omega.$
Then we have $p + \varepsilon v \in \Omega$ or $p - \varepsilon v \in \Omega$ for small enough $\varepsilon > 0.$
Since the same argument goes through, we focus on the first case.
Consider the map $f : [0, \bar \varepsilon) \ni \varepsilon \mapsto u (p + \varepsilon v)$ for small $\bar \varepsilon > 0.$
Observe that $f^{\prime \prime} (0+)$ exists since $u \in C^2 (\bar \Omega).$
Suppose for a contradiction that $f^{\prime \prime} (0 +) > 0.$
Then for small $\varepsilon > 0,$ we have $f^\prime (\varepsilon) = f^\prime (0+) + \int_0^\varepsilon f^{\prime\prime} (\delta) d \delta > f^\prime (0 +) = 0,$ where the last equality holds because $\nabla u (p) = 0.$
Hence, it holds that $f (\varepsilon) = f (0) + \int_0^\varepsilon f^\prime (\delta) d \delta > f (0).$
Since $f (\varepsilon) = u (p + \varepsilon v)$ and $f (0) = u (p),$ this inequality contradicts the fact that $u$ attains its maximum at $p.$
Therefore, we have $f^{\prime\prime} (0 +) \leq 0,$ which implies $v^\prime D^2 u (p) v = f^{\prime\prime} (0 +) \leq 0.$

Next, suppose that $v \in T_p \partial \Omega.$
Recall from Section 6.2 of \cite{gilbarg1977elliptic} and Definition 2.2 of \cite{nardi2015schauder} that there is a diffeomorphism $\varphi_p$ around $p$ that ``straightens'' $\partial \Omega.$
There exists a $C^2$-path $c_\varepsilon \in \bR^d$ such that $\varphi_p (c_0) = p,$ $\varphi_p (c_\varepsilon) \in \Omega$ and $v = \lim_{\varepsilon \to 0} \frac{u (\varphi_p (c_\varepsilon)) - u (p)}{\varepsilon}.$
The rest of proof runs similarly for $f : [0, \bar \varepsilon) \ni \varepsilon \mapsto u (\varphi_p (c_\varepsilon)).$
\end{proof}

\begin{corollary} \label{cor:maximum-principle}
Let $\Omega \subset \bR^d$ be a $C^2$-domain.
If $u \in C^2 (\bar \Omega)$ satisfies $u (p) = \max_{\bar \Omega} u \ (\text{resp.} \ \min_{\bar \Omega} u)$ and $\nabla u (p) = 0$ for $p \in \partial \Omega,$ then $\nabla \cdot (A \nabla u) (p) \leq 0 \ (\text{resp.} \ \geq 0).$
\end{corollary}

\begin{proof}
Recall the eigendecomposition $A_S = \sum_i \lambda_i q_i q_i^\prime$ where $\lambda_i > 0.$
As $\nabla u (p) = 0,$ we have
$$
    \nabla \cdot (A \nabla u) (p)
    =
    \tr (A_S (p) D^2 u (p))
    =
    \sum_i
    \lambda_i (p)
    q_i (p)^\prime (D^2 u (p)) q_i (p)
    .
$$
The conclusion holds by Lemma \ref{lem:second-derivative-is-definite}.
\end{proof}

\begin{lemma} \label{lem:estimate}
Let $\Omega \subset \bR^d$ be a $C^2$-domain.
For $f \in C^{0, \alpha} (\bar \Omega),$ suppose that $u \in C^{2, \alpha} (\bar \Omega)$ is a solution to
$$
    \begin{cases}
        \nabla \cdot (A \nabla u) - u = f \text{ in } \Omega \\
        \ip{A \nabla u}{\vn} = 0 \text{ on } \partial \Omega
    \end{cases}
    .
$$
Then $\norm{u}_{C^0} \leq \norm{f}_{C^0}.$
\end{lemma}

\begin{proof}
Let $p \in \bar \Omega$ such that $|u (p)| = \max_{\bar \Omega} |u|.$

First, suppose that $p \in \Omega,$ and that $u (p) = \max_{\bar \Omega} |u|.$
The same argument works for the case of $u (p) = - \max_{\bar \Omega} |u|$ too.
Observe that $u (p) \geq 0.$
Since $p$ is an interior maximizer, $\nabla u (p) = 0$ holds, and $D^2 u (p)$ is negative semi-definite.
Therefore, we have
$$
    \nabla \cdot (A \nabla u) (p)
    =
    \tr (A_S (p) D^2 u (p))
    =
    \sum_i \lambda_i (p) q_i (p)^\prime D^2 u (p) q_i (p)
    \leq
    0
    ,
$$
where $\lambda_i (p) > 0.$
Thus, $0 \leq u (p) = \nabla \cdot (A \nabla u) (p) - f (p) \leq - f (p),$ and consequently, $\norm{u}_{C^0} \leq \norm{f}_{C^0}.$

Next, suppose that $p \in \partial \Omega$ and that $u (p) = \max_{\bar \Omega} |u|.$
Since $p$ is a maximizer of $u|_{\partial \Omega}$ in particular, $\ip{\nabla u (p)}{\tau} = 0$ holds for $\tau \in T_p \partial \Omega.$
Recall also that $0 = \ip{A (p) \nabla u (p)}{\vn (p)} = \ip{\nabla u (p)}{A (p)^\prime \vn (p)}.$
Since $\ip{\vn (p)}{A (p)^\prime \vn (p)} \geq \lambda > 0$ by the uniform ellipticity of $A,$ we have $A (p)^\prime \vn (p) \notin T_p \partial \Omega.$
By these, we observe that $T_p \partial \Omega$ and $A (p)^\prime \vn (p)$ linearly spans $\bR^d.$
Hence, $\nabla u (p) = 0$ holds.
By Corollary \ref{cor:maximum-principle}, $0 \leq u (p) = \nabla \cdot (A \nabla u) (p) - f (p) \leq - f (p),$ which implies $\norm{u}_{C^0} \leq \norm{f}_{C^0}.$
\end{proof}

Now, we are ready to prove the existence and uniqueness result.

\begin{theorem}
Let $\Omega \subset \bR^d$ be a $C^{2, \alpha}$-domain.
Suppose that $f \in C^{0, \alpha} (\bar \Omega)$ and $g \in C^{1, \alpha} (\bar \Omega)$ satisfy the compatibility condition
\begin{equation} \label{eq:compatibility-condition}
    \int_\Omega f
    =
    \int_{\partial \Omega} g
    .
\end{equation}
Then the problem (\ref{eq:main-pde}) admits a unique solution in the space
$$
    \cC
    \coloneqq
    \left\{
        u \in C^{2, \alpha} (\bar \Omega)
        \mid
        \int_\Omega u = 0
    \right\}
    .
$$
\end{theorem}

\begin{proof}
First consider the following problem
\begin{equation} \label{eq:augumented-pde}
    \begin{cases}
        \nabla \cdot (A \nabla u) - u = f \text{ in } \Omega \\
        \ip{A \nabla u}{\vn} = g \text{ on } \partial \Omega
    \end{cases}
    .
\end{equation}
Since the solution to this PDE for $f = g = 0$ is unique, it has a unique solution in $C^{2, \alpha} (\bar \Omega)$ for any $(f, g) \in C^{0, \alpha} (\bar \Omega) \times C^{1, \alpha} (\bar \Omega)$ by Theorem 5.1 of \cite{nardi2015schauder}.
(See also page 130 of \cite{gilbarg1977elliptic}.)
Hence, the map $C^{0, \alpha} (\bar \Omega) \times C^{1, \alpha} (\bar \Omega) \ni (f, g) \mapsto u \in C^{2, \alpha} (\bar \Omega),$ which assigns the solution $u$ to (\ref{eq:augumented-pde}) for each $(f, g),$ is well-defined.
Restricting this map to the space
$$
    \cA
    \coloneqq
    \left\{
        (f, g)
        \in
        C^{0, \alpha} (\bar \Omega) \times C^{1, \alpha} (\bar \Omega)
        \mid
        \int_\Omega f
        =
        \int_{\partial \Omega} g
    \right\}
    ,
$$
we define 
$$
    \cU : \cA \ni (f, g) \mapsto u \in \cC.
$$
This map is well-defined because for $(f, g) \in \cA,$ the corresponding solution $u$ satisfies
$$
    \int_\Omega u
    =
    \int_\Omega \nabla \cdot (A \nabla u)
    -
    \int_\Omega f
    =
    \int_{\partial \Omega} \ip{A \nabla u}{\vn}
    -
    \int_\Omega f
    =
    \int_{\partial \Omega} g
    -
    \int_\Omega f
    =
    0
    .
$$
Observe also that the map $\cU$ is bijective.
Indeed, it is injective since each $u \in \cC$ uniquely recovers $(f, g) \in \cA$ through (\ref{eq:augumented-pde}), and it is surjective because for $u \in \cC,$ the pair, $f = \nabla \cdot (A \nabla u) - u$ and an $C^{1, \alpha}$-extension $g$ of $g \mid_{\partial \Omega} = \ip{A \nabla u}{\vn},$ is in $\cA.$

For the space
$$
    \cF
    \coloneqq
    \left\{
        f \in C^{0, \alpha} (\bar \Omega)
        \mid
        \int_\Omega f = 0
    \right\}
    ,
$$
consider the map 
$$
    \tilde T : \cF \ni f \mapsto \cU (-f, 0) \in \cC.
$$
Notice that this map is well-defined because $(-f, 0) \in \cA$ for $f \in \cF.$
Since $u = \tilde T f,$ where $f \in \cF,$ solves the PDE
$$
    \begin{cases}
        \nabla \cdot (A \nabla u) - u = - f \text{ in } \Omega \\
        \ip{A \nabla u}{\vn} = 0 \text{ on } \partial \Omega
    \end{cases}
    ,
$$
we have the estimate
$$
    \norm{u}_{C^{2, \alpha}}
    \leq
    C 
    (\norm{u}_{C^0} + \norm{f}_{C^{0, \alpha}})
$$
by Theorem 6.30 of \cite{gilbarg1977elliptic}.
Combined with Lemma \ref{lem:estimate}, this yields
$$
    \norm{u}_{C^{2, \alpha}}
    \leq
    C 
    \norm{f}_{C^{0, \alpha}}
    ,
$$
which implies that $\tilde T$ is a bounded operator.
Recall also that the embedding $\iota : C^{2, \alpha} (\bar \Omega) \hookrightarrow C^{0, \alpha} (\bar \Omega)$ is compact by Lemma 6.36 of \cite{gilbarg1977elliptic}.
Hence, the map 
$$
    T \coloneqq \iota \circ \tilde T : \cF \to \cF
$$ 
is a compact operator.

We shall apply the Fredholm alternative to $I - T : \cF \to \cF.$
Consider the equation $(I - T) u = 0$ for $u \in \cF.$
Since $u = Tu = \tilde T u = \cU (- u, 0) \in C^{2, \alpha} (\bar \Omega),$ this equation is equivalent to $u$ solving
$$
    \begin{cases}
        \nabla \cdot (A \nabla u) = 0 \text{ in } \Omega \\
        \ip{A \nabla u}{\vn} = 0 \text{ on } \partial \Omega
    \end{cases}
    ,
$$
which admits only the trivial solution.
By the Fredholm alternative, for any $v \in \cF,$ there uniquely exists $u \in \cF$ such that $(I - T) u = v.$
In particular, for any $(f, g) \in \cA,$ there uniquely exists $u \in \cF$ such that $(I - T) u = \cU (f, g).$
Since $u = \tilde T u + \cU (f, g) \in \cC$ holds, $u$ uniquely solves the PDE (\ref{eq:main-pde}) in $\cC.$
\end{proof}

\section{Schauder Estimate}

Next, we provide a Schauder-type estimate of the solution to PDE (\ref{eq:main-pde}), following the proof of Theorem 4.1 of \cite{nardi2015schauder}.

\begin{theorem} \label{thm:schauder-estimate}
Let $\Omega \subset \bR^d$ be a $C^{2, \alpha}$-domain.
Suppose that $f \in C^{0, \alpha} (\bar \Omega)$ and $g \in C^{1, \alpha} (\bar \Omega)$ satisfy the compatibility condition (\ref{eq:compatibility-condition}).
Then, the unique solution $u \in \cC$ to the PDE (\ref{eq:main-pde}) satisfies
$$
    \norm{u}_{C^{2, \alpha}}
    \leq
    C 
    \left(\norm{f}_{C^{0, \alpha}} + \norm{g}_{C^{1, \alpha}}\right)
    .
$$
\end{theorem}

\begin{proof}
To obtain the estimate, assume otherwise, i.e., for each $k \in \bN,$ there is $(f_k, g_k) \in C^{0, \alpha} (\bar \Omega) \times C^{1, \alpha} (\bar \Omega)$ satisfying (\ref{eq:compatibility-condition}) such that the corresponding solution $u_k \in \cC$ to (\ref{eq:main-pde}) satisfies
\begin{equation} \label{eq:ineq}
    \norm{u_k}_{C^{2, \alpha}}
    >
    k 
    \left(\norm{f_k}_{C^{0, \alpha}} + \norm{g_k}_{C^{1, \alpha}}\right)
    .
\end{equation}
By rescaling $(f_k, g_k),$ we assume $\norm{u_k}_{C^{2, \alpha}} = 1$ without loss of generality, since $u_k \neq 0.$
By the Ascoli-Arzel\`a theorem, $u_k$ converges to a function $u$ in $\cC$ up to subsequence.
Moreover, (\ref{eq:ineq}), combined with the normalization $\norm{u_k}_{C^{2, \alpha}} = 1,$ implies $(f_k, g_k) \to (0, 0)$ in $C^{0, \alpha} (\bar \Omega) \times C^{1, \alpha} (\bar \Omega).$
Taking the limit $k \to \infty$ along the subsequence, we obtain
$$
    \begin{cases}
        \nabla \cdot (A \nabla u) = 0 \text{ in } \Omega \\
        \ip{A \nabla u}{\vn} = 0 \text{ on } \partial \Omega
    \end{cases}
    ,
$$
which implies $u = 0.$
However, this yields
$$
    0
    =
    \norm{u}_{C^{2, \alpha}}
    =
    \lim_{k \to \infty} \norm{u_k}_{C^{2, \alpha}}
    =
    1
    ,
$$
which is a contradiction.
Hence, the desired estimate holds.
\end{proof}

\section{Dependence on Parameters}

Finally, we show the stability of PDE (\ref{eq:main-pde}) to perturbations in the data.

\begin{theorem} \label{thm:continuious-dependence}
For $0 \leq \varepsilon \ll 1,$ consider functions $f_\varepsilon \in C^{0, \alpha} (\bar \Omega),$ $g_\varepsilon \in C^{1, \alpha} (\bar \Omega),$ and $a_\varepsilon^{ij} \in C^{1, \alpha} (\bar \Omega).$
Suppose that they satisfy
$
    \norm{f_\varepsilon}_{C^{0, \alpha} (\bar \Omega)} < f_M
    ,
    \norm{g_\varepsilon}_{C^{1, \alpha} (\bar \Omega)} < g_M
    ,
    \norm{a_\varepsilon^{ij}}_{C^{1, \alpha} (\bar \Omega)} < a_M
    ,
$
and
$
    \lim_{\varepsilon \searrow 0} \norm{f_\varepsilon - f_0}_{C^{0, \alpha} (\bar \Omega)} = 0
    ,
    \lim_{\varepsilon \searrow 0} \norm{g_\varepsilon - g_0}_{C^{1, \alpha} (\bar \Omega)} = 0
    ,
    \lim_{\varepsilon \searrow 0} \norm{a_\varepsilon^{ij} - a_0^{ij}}_{C^{1, \alpha} (\bar \Omega)} = 0
    .
$
Let $u_\varepsilon \in C^{2, \alpha} (\bar \Omega)$ be the solution to the PDE
\begin{equation}
    \begin{cases}
        \nabla \cdot (A_\varepsilon \nabla u) = f_\varepsilon \text{ in } \Omega \\
        \ip{A_\varepsilon \nabla u}{\vn} = g_\varepsilon \text{ on } \partial \Omega
    \end{cases}
    .
\end{equation}
Then $\lim_{\varepsilon \searrow 0} \norm{u_\varepsilon - u_0}_{C^{2, \alpha} (\bar \Omega)} \to 0$ holds.
\end{theorem}

\begin{proof}
By Theorem \ref{thm:schauder-estimate}, it holds
\begin{align*}
    &\quad
    \norm{u_\varepsilon - u_0}_{C^{2, \alpha} (\bar \Omega)}
    \\
    &\leq
    C
    \left(
        \norm{f_\varepsilon - f_0 - \nabla \cdot ((A_\varepsilon - A_0) \nabla u_\varepsilon)}_{C^{0, \alpha} (\bar \Omega)}
        +
        \norm{g_\varepsilon - g_0 - \ip{(A_\varepsilon - A_0) \nabla u_\varepsilon}{\vn}}_{C^{1, \alpha} (\bar \Omega)}
    \right)
    ,
\end{align*}
where the coefficient can be taken identically since $\norm{a_\varepsilon^{ij}}_{C^{1, \alpha} (\bar \Omega)}$ is bounded.
Combining the conditions in the statement with the estimate
$
    \norm{u_\varepsilon}_{C^{2, \alpha} (\bar \Omega)}
    \leq
    C
    (f_M + g_M)
$
yields the conclusion.
\end{proof}

The following is shown in a similar way.
\begin{theorem}
Under the same setup as Theorem \ref{thm:continuious-dependence}, assume also that there exist functions $D_\varepsilon f_0 \in C^{0, \alpha} (\bar \Omega),$ $D_\varepsilon g_0 \in C^{1, \alpha} (\bar \Omega),$ and $D_\varepsilon a_0^{ij} \in C^{1, \alpha} (\bar \Omega)$ such that $
    \lim_{\varepsilon \searrow 0} \norm{\frac{f_\varepsilon - f_0}{\varepsilon} - D_\varepsilon f_0}_{C^{0, \alpha} (\bar \Omega)} = 0
    ,
$
$
    \lim_{\varepsilon \searrow 0} \norm{\frac{g_\varepsilon - g_0}{\varepsilon} - D_\varepsilon g_0}_{C^{1, \alpha} (\bar \Omega)} = 0
    ,
$
and
$
    \lim_{\varepsilon \searrow 0} \norm{\frac{A_\varepsilon - A_0}{\varepsilon} - D_\varepsilon A_0}_{C^{1, \alpha} (\bar \Omega)} = 0
$
hold.
Then there exists $D_\varepsilon u_0 \in C^{2, \alpha} (\bar \Omega)$ such that 
$
    \lim_{\varepsilon \searrow 0} \norm{\frac{u_\varepsilon - u_0}{\varepsilon} - D_\varepsilon u_0}_{C^{2, \alpha} (\bar \Omega)} = 0
$
holds.
\end{theorem}

\printbibliography

\end{document}